\documentclass[11pt,twoside,reqno]{amsart}
\usepackage{amssymb}
\usepackage{enumerate}
\usepackage{amsmath}
\usepackage{a4wide}
\usepackage[usenames]{color}
\usepackage[dvipsnames]{xcolor}
\usepackage{soul}
\usepackage[all]{xy}
\usepackage{array}
\usepackage{graphicx}
\usepackage{tabularx}
\usepackage[table]{xcolor}
\usepackage{booktabs}
\usepackage{verbatim}
\usepackage[hidelinks=true]{hyperref} 

\usepackage{listings}

\definecolor{codegreen}{rgb}{0,0.6,0}
\definecolor{codegray}{rgb}{0.5,0.5,0.5}
\definecolor{codepurple}{rgb}{0.58,0,0.82}
\definecolor{backcolour}{rgb}{0.95,0.95,0.92}

\lstdefinestyle{Pstystyle}{
    backgroundcolor=\color{backcolour},   
    commentstyle=\color{codegreen},
    keywordstyle=\color{magenta},
    numberstyle=\tiny\color{codegray},
    stringstyle=\color{codepurple},
    basicstyle=\ttfamily\footnotesize,
    breakatwhitespace=false,         
    breaklines=true,                 
    captionpos=b,                    
    keepspaces=true,                 
    numbers=left,                    
    numbersep=5pt,                  
    showspaces=false,                
    showstringspaces=false,
    showtabs=false,                  
    tabsize=2
}

\usepackage{mathtools}

\usepackage{lmodern}
\usepackage[utf8]{inputenc}
\usepackage[L7x]{fontenc}  

\definecolor{blanchedalmond}{rgb}{1.0, 0.92, 0.8}

\usepackage{array}

\usepackage[
   backend=bibtex, 
   style=numeric,
   isbn=false,
   giveninits=true,
   doi=false,
   url=false]{biblatex}
   
\newtheorem{theorem}{Theorem}

\newtheorem{lemma}[theorem]{Lemma}

\newtheorem{problem}[theorem]{Problem}

\theoremstyle{remark}
\newtheorem{remark}[theorem]{Remark}

\def\leq{\leqslant} \def\geq{\geqslant} \def\al{\alpha}

\begin{document}

\title{Linear relations of four conjugates of an algebraic number of degree eight}

\author{Žygimantas Baronėnas, Paulius Drungilas, and Jonas Jankauskas}

\address{Institute of Mathematics, Faculty of Mathematics and Informatics, Vilnius
University, Naugarduko 24, Vilnius LT-03225, Lithuania}
\email{zygimantas.baronenas@mif.stud.vu.lt}

\address{Institute of Mathematics, Faculty of Mathematics and Informatics, Vilnius
University, Naugarduko 24, Vilnius LT-03225, Lithuania}
\email{paulius.drungilas@mif.vu.lt}

\address{Institute of Mathematics, Faculty of Mathematics and Informatics, Vilnius
University, Naugarduko 24, Vilnius LT-03225, Lithuania}
\email{jonas.jankauskas@mif.vu.lt}


\begin{abstract}
We characterize all algebraic numbers $\alpha$ of degree $8$ for which there exist four distinct algebraic conjugates $\alpha_1$, $\alpha_2$, $\alpha_3$, $\alpha_4$ of $\alpha$ satisfying the linear relation $\alpha_{1}=\alpha_{2}+\alpha_{3}+\alpha_{4}$. 
Analogous characterization is obtained for the linear relation $\alpha_{1}+\alpha_{2}+\alpha_{3}+\alpha_{4}=0$. 
In particular, when an algebraic number $\alpha$ of degree $8$ has a non-even minimal polynomial and possesses exactly six distinct 
linear relations of the form $\alpha_{i_1}+\alpha_{i_2}+\alpha_{i_3}+\alpha_{i_4}=0$, we prove that $\alpha$ 
is a sum of a quadratic and a quartic algebraic number. 
\end{abstract}
\maketitle

\section{Introduction}\label{intro}

In this paper we are interested in linear relations among the algebraic conjugates of an algebraic number. 
More precisely, we are interested in the following problem:

\begin{problem}\label{prb1}
Suppose that $a_1,a_2,\dotsc,a_n$ are non-zero integers. Find all algebraic numbers $\alpha$ which satisfy the linear relation
\begin{equation}\label{eqlr1}
a_1\alpha_1+a_2\alpha_2+\dotsb+a_n\alpha_n=0,
\end{equation}
where $\alpha_1,\alpha_2,\dotsc,\alpha_n$ are distinct (not necessarily all) algebraic conjugates of $\alpha$.
\end{problem}

Let $\alpha$ be an algebraic number of degree $d$ over the field of rational numbers $\mathbb{Q}$ whose minimal polynomial is
\begin{equation*}\label{eqmpol1}
p(x)=x^d+c_{d-1}x^{d-1}+\dotsb+c_1x+c_0,
\end{equation*}
where $c_0,c_1,\dotsc,c_{d-1}\in\mathbb{Q}$. 
Let $\alpha_{1}:=\alpha, \alpha_{2}, \dots, \alpha_{d}$ be the algebraic conjugates of $\alpha$, 
so that  $\alpha_{1}, \alpha_{2}, \dots, \alpha_{d}$ are all the roots of $p(x)$, which in turn 
is an irreducible polynomial in $\mathbb{Q}[x]$. 
The number $tr(\alpha):=\alpha_1+\alpha_2+\dotsb+\alpha_d$ is called the \textit{trace} (or the \textit{absolute trace}) of $\alpha$. 
Vieta's formulas yield $\alpha_1+\alpha_2+\dotsb+\alpha_d=-c_{d-1}$. 
So that if $tr(\alpha)=-c_{d-1}=0$, then we have the linear relation $\alpha_1+\alpha_2+\dotsb+\alpha_d=0$ 
with $n=d$ and $a_1=a_2=\dotsc =a_{d}=1$ in \eqref{eqlr1}. 
The linear relation \eqref{eqlr1} for an algebraic number $\alpha$ of degree $d$ is called \textit{trivial} if 
$n=d$ and $a_1=a_2=\dotsc =a_{d}$. Conversely, if the linear relation \eqref{eqlr1} for an algebraic number $\alpha$ is trivial, then $tr(\alpha)=0$.  

As far as we know, the first general result in the context of Problem~\ref{prb1} was obtained by Kurbatov \cite{Kurbatov1977}, 
who proved that no algebraic number of prime degree possesses a non-trivial linear relation \eqref{eqlr1}. 
Moreover, Baron et al \cite{Baron1995} (see also \cite{Dixon1997}) proved that if the Galois group of the normal closure 
of the number field generated by an algebraic number $\alpha$ acts 2-transitively on the set of 
algebraic conjugates of $\alpha$, then there are no non-trivial linear relations \eqref{eqlr1} for $\alpha$.

Smyth \cite{Smyth1986additive} investigated a similar problem of finding all tuples 
$(a_1,a_2,\dotsc,a_n)$, $gcd(a_1,a_2,\dotsc,a_n)=1$, of non-zero integers which satisfy \eqref{eqlr1} for some non-zero algebraic 
conjugate numbers $\alpha_1,\alpha_2,\dotsc,\alpha_n$ (not necessarily distinct). 
He proved that any such tuple satisfies $|a_i|\leq \sum_{j\neq i} |a_j|$ for every $i\in\{1,2,\dotsc,n\}$, 
and for every prime number $p$, at most $n-2$ numbers $a_i$ are divisible by $p$.  
He also conjectured that this necessary condition is sufficient. 
This conjecture was recently proved by Ellenberg and Hardt \cite{EllenbergHardt2025} (generalization of Smith's conjecture to global fields can be found in \cite{HardtYin2022}).

Girstmair in \cite{Girstmair1999} proposed a theoretical framework to study linear and multiplicative relations \eqref{eqlr1} for $\alpha$, based on representation theory of finite groups, applied to the Galois group of the normal closure of $\mathbb{Q}(\alpha)$. Among other things, he proved that there exists an algebraic number $\alpha$ of degree 9 whose distinct algebraic conjugates satisfy the linear relation
\[
4\alpha_1 + \alpha_2 + \alpha_3 + \alpha_4 + \alpha_5 - 2\alpha_6 - 2\alpha_7 - 2\alpha_8 - 2\alpha_9 = 0.
\]

Dubickas \cite{Dubickas2002} (see also \cite[Theorem~3']{Dixon1997}) proved that the linear relation \eqref{eqlr1} can't hold if $|a_1|\geq \sum_{i=2}^n|a_i|$ and $n\geq 3$. This generalizes the result obtained by Smyth \cite{Smyth1982} for the linear relation $\alpha_1 \pm \alpha_2 \pm 2\alpha_3=0$.  

One more general result was obtained by Dubickas and Virbalas in \cite{DuVi25}, where they proved that every linear relation implies a corresponding multiplicative relation (several more results on multiplicative relations can be found in \cite{Ferguson1997}). 


Considerable effort has been dedicated to studying Problem~\ref{prb1} for small values of $n$, $\deg(\alpha)$ and $a_1,a_2,\dotsc,a_n$. 
Consider a non-trivial linear relation \eqref{eqlr1} with $n=2$. 
One can easily show that it is equivalent to $\alpha_1+\alpha_2=0$, which implies that 
the minimal polynomial $p(x)$ of $\alpha$ is an even polynomial. 
Conversely, if $p(x)$ is an even polynomial, then some two distinct algebraic conjugates 
of $\alpha$ satisfy $\alpha_1+\alpha_2=0$. 
This is the only case ($n=2$) in which Problem~\ref{prb1} has been completely solved. 

When $n=3$ in \eqref{eqlr1}, we have two relations of the simplest form: $\alpha_1+\alpha_2=\alpha_3$ and $\alpha_1+\alpha_2+\alpha_3=0$. 
Girstmair \cite[Proposition~9]{Girstmair1999} (see also \cite[Theorem~5]{Dixon1997}) proved that if 
$G$ is a finite abelian group whose order $|G|$ is divisible by 6, then there exists an 
irreducible polynomial with Galois group $G$ whose three distinct roots satisfy $\alpha_1+\alpha_2=\alpha_3$.  
He also proved that $6\mid |G|$ is a necessary condition in this case. 
See \cite{Girstmair1982, Girstmair2006, Girstmair2007, Girstmair2008, Lalande2007, Lalande2010, DubickasHareJankauskas2017} for other results related to this linear relation.  
On the other hand, Dubickas and Jankauskas \cite{DubickasJankauskas2015} proved that if the linear relation $\alpha_1+\alpha_2=\alpha_3$ holds for $\alpha$ of degree $d\leq 8$, then $d=6$. 
Moreover, they proved that this linear relation holds for a sextic algebraic number $\alpha$ if and 
only if $\alpha$ is a root of an irreducible polynomial of the form 
\[
x^6+2ax^4+a^2x^2+b\in\mathbb{Q}[x].
\]
It was proved in \cite{DubickasJankauskas2015} that the linear relation $\alpha_1+\alpha_2+\alpha_3=0$ holds for $\alpha$ of degree $4\leq d \leq 8$ if and only if $d=6$ and the minimal polynomial of $\alpha$ is of the form
\[
p(x) = x^{6} + 2ax^{4} + 2bx^{3} + (a^{2} - c^{2}t)x^{2} + 2(ab - cet)x + b^{2} - e^{2}t
\]
for some rational numbers $a, b, c, e\in\mathbb{Q}$ and some square-free integer $t\in\mathbb{Z}$. 
Virbalas \cite{Virbalas2025} determined all possible linear relations \eqref{eqlr1} in case when $n=3$ 
and $\deg(\alpha)\leq 8$. He also described all possible Galois groups of such algebraic numbers $\alpha$. 
Moreover, in \cite{Virbalas2025} Virbalas proved that no algebraic number of degree $2p$, where $p\geq 5$ is a prime number, satisfies the linear relation $\alpha_1+\alpha_2+\alpha_3=0$. 

Dubickas and Virbalas in \cite{DuVi25} determined all possible linear relations \eqref{eqlr1} for quartic algebraic numbers $\alpha$ (see also \cite{Kitaoka2017}). The multiplicative analog was obtained by Serrano Holgado \cite{Serrano2025}.  

In \cite{Baronenas2026} we proved that no algebraic number $\alpha$ of degree $d\in\{4,5,6,7\}$ satisfies 
the linear relation $\alpha_1=\alpha_2+\alpha_3+\alpha_4$. Our first main result extends to $d=8$.

\begin{theorem}\label{t1}
Let $\alpha$ be an algebraic number of degree $d=8$. Some four distinct algebraic conjugates of $\alpha$ satisfy the relation
\begin{equation*}
\alpha_{1}=\alpha_{2}+\alpha_{3}+\alpha_{4}
\end{equation*}
if and only if the minimal polynomial of $\alpha$ is an irreducible polynomial of the form
\begin{equation*}
x^{8} + ax^{6} + bx^{4} + cx^{2} + \frac{1}{64}(a^{2}-4b)^{2}\in\mathbb{Q}[x].
\end{equation*}
For every such $\alpha$, the Galois group of the normal closure of $\mathbb{Q}(\alpha)$ over $\mathbb{Q}$ is isomorphic to one of the groups, given 
in Table~\ref{table:t1}. Moreover, every such group appears as the Galois group for some $\alpha$ (see Table~\ref{table:t1}).
%
%
\end{theorem}

\begin{table}[h!]
\centering
\begin{tabular}{ |c|c| }
\hline
Polynomial $p(x)$ & Galois group $G$\\
\hline
\hline
$x^8 + 6x^6 + 8x^4 + 3x^2 + 1/4$ &  $C_2\times C_2\times C_2$\\
$x^8 + 7x^4 + 8x^2 + 49/4$ & $C_2\times C_4$\\
$x^8 + x^4 + 4x^2 + 1/4$ &  $D_4$\\
$x^8 + 3x^4 + 4x^2 + 9/4$ & $C_2\times D_4$\\
$x^8 + 4x^6 + 4x^2 + 4$ &  $S_4$\\
$x^8 + 6x^4 + 8x^2 + 9$ &  $C_2\times A_4$\\
$x^8 + x^4 + 2x^2 + 1/4$ & $C_2\times S_4$\\
\hline
\end{tabular}
\vspace{0,2cm}
\caption{Minimal polynomials $p(x)$ of $\alpha$ from Theorem~\ref{t1} with the corresponding Galois groups $G$.}
\label{table:t1}
\end{table}

Another simple linear relation \eqref{eqlr1} for $n=4$ is of the form $\alpha_{1}+\alpha_{2}+\alpha_{3}+\alpha_{4}=0$. 
In \cite{Baronenas2026} we proved that an algebraic number $\alpha$ of degree $d\in\{4,5,6,7\}$ satisfies 
such linear relation if and only if $d=4$ and $tr(\alpha)=0$ or $d=6$ and the minimal polynomial 
of $\alpha$ is an irreducible even polynomial. Our second main result solves the Problem~\ref{prb1}  
in case when $n=4$, $a_1=\dotsc=a_4$ and $\deg(\alpha)=8$.

\begin{theorem}\label{t2}
Let $\alpha$ be an algebraic number of degree $d=8$. Denote by $p(x)$ the minimal polynomial of $\alpha$. Some four distinct algebraic conjugates of $\alpha$ satisfy the relation
\begin{equation}\label{eqrel4}
\alpha_{1}+\alpha_{2}+\alpha_{3}+\alpha_{4}=0
\end{equation}
if and only if $p(x)$ is an irreducible polynomial and $p(x)\in\{ p_1(x),p_2(x),$ $p_3(x), p_4(x)\}$, where
\begin{equation*}
\begin{split}
\quad p_1(x)=\;\,&x^{8} + 2ax^{6} + 2cx^{5} + (a^{2} - b^{2}t + 2e) x^{4} + 2(ac - bdt)x^{3}\\ 
                &+ (2(ae - bft) + c^{2} - d^{2}t)x^{2} + 2(ce - dft)x + (e^{2}-f^{2}t),\\[2mm]
\quad p_2(x)=\;\,&x^{8}-2(a^{2}t-b)x^{6}+2actx^{5}+( a^{4}t^2 -2a^{2}bt+b^{2}-c^{2}t+2d)x^{4}\\
                 &-2at(a^{2}ct-2e)x^{3}+( a^2c^2t^2-a^2b^2t+2a^2dt-2cet+2bd)x^{2}\\
                 &+2at(be-cd)x+d^{2}-e^{2}t,\\[2mm]
\quad p_3(x)=\;\,&x^{8} + (-4 a + 2 b) x^{6} + 2 c x^{5} + (6 a^{2} - 2 a b + b^{2} + 2 d) x^{4}\\ 
                  &+(4 a c + 2 b c) x^{3}+ (-4 a^{3} - 2 a^{2} b - 2 a b^{2} + 12 a d + c^{2} + 2 b d) x^{2}\\ 
                  &-2c(3a^{2} + a b  - d) x+ a^{4} + 2 a^{3} b + a^{2} b^{2} + 2 a^{2} d - a c^{2} + 2 a b d + d^{2},\\[2mm]
\quad p_4(x)=\;\,&x^{8}+ax^{6}+bx^{4}+cx^{2}+d,
\end{split}
\end{equation*}
here $a, b, c, d, e, f\in\mathbb{Q}$ and $t$ is a square-free integer.
\end{theorem}

Let $\alpha$ be an algebraic number of degree 8 which satisfies the linear relation \eqref{eqrel4}. 
Two such relations $\alpha_{i_1}+\alpha_{i_2}+\alpha_{i_3}+\alpha_{i_4}=0$ and $\alpha_{j_1}+\alpha_{j_2}+\alpha_{j_3}+\alpha_{j_4}=0$ 
are called distinct, if 
\[
\{\alpha_{i_1},\alpha_{i_2},\alpha_{i_3},\alpha_{i_4}\}\neq \{\alpha_{j_1},\alpha_{j_2},\alpha_{j_3},\alpha_{j_4}\}.
\]
In the proof of Theorem~\ref{t2}, we consider all possible distinct linear relations  $\alpha_{i_1}+\alpha_{i_2}+\alpha_{i_3}+\alpha_{i_4}=0$ for $\alpha$. 
Let $N$ be the number of such relations. 
We prove that $N=2,4,6$ or 8 and consider each case separately. 
Polynomials $p_1(x),p_2(x)$ and $p_3(x)$ (see Theorem~\ref{t2}) correspond to $N=2,4$ and 6, respectively. 
Moreover, it follows that if $N=6$ and the minimal polynomial $p(x)$ is a non-even polynomial, then  $\alpha$ 
is a sum of a quadratic and a quartic algebraic number (see Case $N=6$ in the proof of Theorem~\ref{t2}).

The paper is organized as follows. Some auxiliary results are given in Section~\ref{intro2} and the proofs of Theorem~\ref{t1} and \ref{t2} are provided in Section~\ref{prooft1} and \ref{prooft2}, respectively. The code of computations with SageMath~\cite{sagemath}, used in the proofs of Theorem~\ref{t1} and \ref{t2}, are given in the Appendix.

\section{Auxiliary results}\label{intro2}



Smyth \cite{Smyth1982} proved the following result, which is often used to investigate linear relations of conjugate algebraic numbers.

\begin{lemma}\label{intro7}
If $\alpha_{1}, \alpha_{2}, \alpha_{3}$ are three conjugates of an algebraic number satisfying $\alpha_{1}\neq\alpha_{2}$ then $2\alpha_{1}\neq \alpha_{2} + \alpha_{3}$.
\end{lemma}

Dubickas \cite{Dubickas2002} generalized Lemma $\ref{intro7}$ by proving the following 

\begin{lemma}\label{lemma4}
If $\beta_{1}, \beta_{2},\dots , \beta_{n}$, where $n\geq 3$, are distinct algebraic numbers conjugate over a field of characteristic zero $K$ and $k_{1}, k_{2},\dots ,k_{n}$ are non-zero rational numbers satisfying $|k_{1}| \geq |k_{2}|+\dots+|k_{n}|$ then
\begin{equation*}
k_{1}\beta_{1} + k_{2}\beta_{2} +\dots+ k_{n}\beta_{n}\notin K.
\end{equation*}
\end{lemma}

In the proof of Theorem~\ref{t1} we will need one more result, 
obtained by Dubickas and Jankauskas \cite{DubickasJankauskas2015}. 

\begin{lemma}\label{intro3}
The equality
\begin{equation*}
k_{1}\alpha_{1}+k_{2}\alpha_{2}+\cdots+k_{d}\alpha_{d} = 0
\end{equation*}
with conjugates $\alpha_{1}, \alpha_{2},\dots, \alpha_{d}$ of an algebraic number $\alpha$ of degree $d$ over $\mathbb{Q}$ and $k_{1}, k_{2},\dots, k_{d}\in\mathbb{Z}$ satisfying $\sum_{i=1}^{d} k_{i}\neq 0$ can only hold if $tr(\alpha) := \alpha_{1} + \alpha_{2} + \cdots + \alpha_{d} = 0$.
\end{lemma}

\section{Proof of Theorem~\ref{t1}}\label{prooft1}

\begin{proof}[Proof of Theorem~\ref{t1}]
Let $\alpha$ be an algebraic number of degree $d=8$ and let $\alpha_1=\alpha,\alpha_2,\dotsc, \alpha_8$ be the algebraic conjugates of $\alpha$. Assume that
\begin{equation}\label{eq1}
\alpha_{1}=\alpha_{2}+\alpha_{3}+\alpha_{4}.
\end{equation}
Throughout the proof, we consider only the relations $\alpha_i=\alpha_{i_1}+\alpha_{i_2}+\alpha_{i_3}$, where $\alpha_{i},\alpha_{i_1},\alpha_{i_2},\alpha_{i_3}$ are four distinct algebraic sonjugates of $\alpha$. 
We say that two relations $\alpha_i=\alpha_{i_1}+\alpha_{i_2}+\alpha_{i_3}$ and $\alpha_j=\alpha_{j_1}+\alpha_{j_2}+\alpha_{j_3}$ are distinct if $\{\alpha_{i_1},\alpha_{i_2},\alpha_{i_3}\}\neq \{\alpha_{j_1},\alpha_{j_2},\alpha_{j_3}\}$.

Lemma $\ref{intro3}$, in view or \eqref{eq1},  implies that $tr(\alpha)=\alpha_{1}+\alpha_{2}+\dotsb+\alpha_{8}=0$. 

\begin{lemma}\label{le316}
For every $i\in\{1,2,\dotsc,8\}$ there exists exactly one relation of the form 
$\alpha_i=\alpha_{i_1}+\alpha_{i_2}+\alpha_{i_3}$.
\end{lemma}

\begin{proof}
It is sufficient to prove this claim for $i=1$, since if 
some $\alpha_i$ has at least two distinct relations, then, applying an automorphism $\phi\in G$ which sends $\alpha_i$ to $\alpha_1$, we obtain that there are at least two distinct relations for $\alpha_1$ as well. Indeed, assume that  we have one more relation $\alpha_{1} = \alpha_{k}+\alpha_{l}+\alpha_{m}$ which 
is distinct from \eqref{eq1}. If $\{\alpha_{k},\alpha_{l},\alpha_{m}\}\cap\{\alpha_{2},\alpha_{3},\alpha_{4}\} = \varnothing$, then, without loss of generality, say, $\alpha_{1}=\alpha_{5}+\alpha_{6}+\alpha_{7}$. 
By adding this relation with \eqref{eq1} and using $tr(\alpha)=0$, we derive $\alpha_{8}=-3\alpha_{1}$. 
This is impossible. Indeed, let $k\in\{1,2,\dotsc,8\}$ be such that $|\alpha_k|=m:=\max\{|\alpha_j|:j=1,2,\dotsc,8\}$ and 
select  an automorphism $\pi\in G$ which sends $\alpha_1$ to $\alpha_k$. 
Then $\pi$ maps $\alpha_{8}=-3\alpha_{1}$ to $\pi(\alpha_{8})=-3\alpha_{k}$. 
Hence, $3m=|-3\alpha_{k}| = |\pi(\alpha_{8})|\leq m$. A contradiction. 
If $|\{\alpha_{k},\alpha_{l},\alpha_{m}\}\cap\{\alpha_{2},\alpha_{3},\alpha_{4}\}| = 1$, then, without loss of generality, say, $\alpha_{1}=\alpha_{2}+\alpha_{5}+\alpha_{6}$. 
By adding this equation with \eqref{eq1} and using $tr(\alpha)=0$, we derive $3\alpha_{1}-\alpha_{2}+\alpha_{7}+\alpha_{8}=0$, which contradicts Lemma $\ref{lemma4}$. 
If $|\{\alpha_{k},\alpha_{l},\alpha_{m}\}\cap\{\alpha_{2},\alpha_{3},\alpha_{4}\}| = 2$, then, without loss of generality, say, $\alpha_{1}=\alpha_{2}+\alpha_{3}+\alpha_{5}$. 
But then, by subtracting this equation from \eqref{eq1}, we obtain $\alpha_{4}=\alpha_{5}$, which is impossible. 
\end{proof}

Let $G$ be the Galois group of the normal closure of $\mathbb{Q}(\alpha)$ over $\mathbb{Q}$. Note that this normal closure is also the splitting field of the minimal polynomial of $\alpha$ over $\mathbb{Q}$, and therefore $G$ is the Galois group of this polynomial. The group $G$ corresponds to some transitive subgroup of the full symmetric group $S_{8}$. Lemma~\ref{le316} implies

\begin{lemma}\label{le341}
There are exactly 8 distinct relations of the form 
$\alpha_i=\alpha_{i_1}+\alpha_{i_2}+\alpha_{i_3}$.
\end{lemma}

\begin{proof}
By acting on \eqref{eq1} with appropriate automorphisms from $G$ we  obtain relations of the form \eqref{eq1} for every algebraic conjugate of $\alpha$:
\begin{equation}\label{eq2}
\begin{split}
\alpha_{1} &= \alpha_{2}\phantom{a} + \alpha_{3}\phantom{a} + \alpha_{4},\\
\alpha_{2} &= \alpha_{i_{21}} + \alpha_{i_{22}} + \alpha_{i_{23}},\\
   \dotsb            &=\;\;\; \dotsb\;\;\; \dotsb\;\;\;\dotsb\\
\alpha_{8} &= \alpha_{i_{81}}+\alpha_{i_{82}} + \alpha_{i_{83}}.
\end{split}
\end{equation}
We have obtained 8 distinct relations in \eqref{eq2}, and, by Lemma~\ref{le316}, these are all possible distinct relations of the form $\alpha_i=\alpha_{i_1}+\alpha_{i_2}+\alpha_{i_3}$.
\end{proof}

With every relation $\alpha_j=\alpha_{i_{j1}}+\alpha_{i_{j2}}+\alpha_{i_{j3}}$  in \eqref{eq2}, which is equivalent to $\alpha_j - \alpha_{i_{j1}} - \alpha_{i_{j2}} - \alpha_{i_{j3}}=0$, 
we associate the vector $v_j=(x_{j1},x_{j2},\dotsc,x_{j8})\in\mathbb{Q}^8$, defined as follows: 
$x_{j}=1$, $x_{i_{j1}}=x_{i_{j2}}=x_{i_{j3}}=-1$ and $x_k=0$ 
for all $k\in\{1,2,\dotsc,8\}\setminus\{j,i_{j1},i_{j2},i_{j3}\}$. 
For example, $v_1=(1,-1,-1,-1,0,0,0,0)$. 
Let $\mathcal{R}$ be the $8\times 8$ matrix whose rows are the vectors $v_1,v_2,\dotsc,v_8$. 
For any $j\in\{2,3,\dots, 8\}$, let $\mathcal{R}_j$ be the $9\times 8$ matrix, whose first 8 rows coincide with respective rows of $\mathcal{R}$ and the last row corresponds to the relation $\alpha_1+\alpha_j=0$, i.e., the last row of $\mathcal{R}_j$ is equal 
to the vector $(1,y_2,y_3,\dotsc,y_8)$, where $y_j=1$ and $y_k=0$ for all $k\in\{2,3,\dotsc, 8\}\setminus\{j\}$.  
Note that if the rank of $\mathcal{R}$ equals the rank of $\mathcal{R}_j$, then some linear combination of 
relations \eqref{eq2} yields the relation $\alpha_1+\alpha_j=0$.

Recall that $G$ is the Galois group of the normal closure of $\mathbb{Q}(\alpha)$ over $\mathbb{Q}$.  
We identify $G$ with the corresponding transitive subgroup of the full symmetric group $S_{8}$. 
So that any automorphism $\phi\in G$ acts on the indices of $\alpha_i$, i.e., 
$\phi(\alpha_i) = \alpha_{\phi(i)}$ for all $i\in\{1,2,\dotsc,8\}$. 

Let $H$ be a transitive subgroup of the symmetric group $S_8$. 
Consider the orbit (under $H$) of the relation $\alpha_1=\alpha_{2}+\alpha_{3}+\alpha_{4}$, which consists of the following relations:
\begin{equation}\label{eq371}
\alpha_{\tau(1)}=\alpha_{\tau(2)}+\alpha_{\tau(3)}+\alpha_{\tau(4)}\;\;\;\text{ for all }\tau\in H.
\end{equation}

If $H=G$, then, by Lemma~\ref{le341}, this orbit consists of exactly 8 distinct relations. 

Let $\mathcal{H}$ be the set of transitive subgroups $H$ of the full symmetric group $S_8$ for which the orbit 
\eqref{eq371} consists of exactly 8 distinct relations. 

We have carried out the following calculation on SageMath~\cite{sagemath} (the code of computations is provided in Listing~\ref{lst:1} in the Appendix). 
There are exactly 3922 subgroups in $\mathcal{H}$. 
For every $H\in\mathcal{H}$ and the corresponding 8 relations in \eqref{eq371} we have checked that 
there exists $j\in\{2,3,\dots, 8\}$ such that the rank of $\mathcal{R}$ equals the rank of $\mathcal{R}_j$. 
Since $G\in\mathcal{H}$, we obtain that there exists $j\in\{2,3,\dots, 8\}$ such that some linear combination of 
relations \eqref{eq2} yields the relation $\alpha_1+\alpha_j=0$. Thus, $\alpha_1=-\alpha_j$. 
Moreover, there are exactly 3538 subgroups $H\in\mathcal{H}$ for which there exist two distinct indices $i,j\in\{2,3,\dots, 8\}$ 
such that the ranks of the three matrices $\mathcal{R}$, $\mathcal{R}_i$ and $\mathcal{R}_j$ are equal. 
This yields two distinct relations $\alpha_1+\alpha_i=0$ and $\alpha_1+\alpha_j=0$ which imply $\alpha_i=\alpha_j$. 
 This is impossible. Hence, the Galois group $G$ is one of the remaining $3922-3538=384$ subgroups in $\mathcal{H}$, 
  each isomorphic to some group, given in Table~\ref{table:t1}.

Since $\alpha_1=-\alpha_j$ for some $j\in\{2,3,\dots, 8\}$, we obtain that the minimal polynomial $p(x)$ of $\alpha$ is an even polynomial, i.e., 
\begin{equation*}
p(x)=x^{8}+ax^{6}+bx^{4}+cx^{2}+d
\end{equation*}
for some $a,b,c,d\in\mathbb{Q}$. We will prove that $d=(a^{2}-4b)^{2}/64$. 
Indeed, let $x_{1}, x_{2}, x_{3}, x_{4}$ be the roots of the polynomial $x^{4}+ax^{3}+bx^{2}+cx+d$. 
Then $\pm\sqrt{x_{1}}, \pm\sqrt{x_{2}}, \pm\sqrt{x_{3}}, \pm\sqrt{x_{4}}$ are the roots of the polynomial $p(x)$. 
Recall that some four distinct roots of $p(x)$ satisfy \eqref{eq1}. 
Therefore, without loss of generality, we can assume that  
$\epsilon_{1}\sqrt{x_{1}}=\epsilon_{2}\sqrt{x_{2}}+\epsilon_{3}\sqrt{x_{3}}+\epsilon_{4}\sqrt{x_{4}}$, 
where $\epsilon_j=\pm 1$, $j=1,2,3,4$.
By squaring both sides of  $\epsilon_{1}\sqrt{x_{1}}-\epsilon_{2}\sqrt{x_{2}}=\epsilon_{3}\sqrt{x_{3}}+\epsilon_{4}\sqrt{x_{4}}$, we get
\begin{equation*}
x_{1}+x_{2}-x_{3}-x_{4}=2\left(\epsilon_{1}\epsilon_{2}\sqrt{x_{1}x_{2}}+\epsilon_{3}\epsilon_{4}\sqrt{x_{3}x_{4}}\right).
\end{equation*}
By squaring both sides again and doing some algebraic manipulations, we derive
\begin{equation*}
\left(\sum_{i=1}^{4}x_{i}\right)^{2}-4\sum_{1\leq i<j\leq 4}x_{i}x_{j}=8\epsilon_{1}\epsilon_{2}\epsilon_{3}\epsilon_{4}\sqrt{x_{1}x_{2}x_{3}x_{4}}.
\end{equation*}
Vieta's formulas imply that $a^{2}-4b = 8\epsilon_{1}\epsilon_{2}\epsilon_{3}\epsilon_{4}\sqrt{d}$. 
Again, by squaring both sides, we obtain $d=(a^{2}-4b)^{2}/64$. Hence, $p(x)$ is of the form, as given in the statement of the theorem.

Conversely, assume that
\[
p(x)=x^{8}+ax^{6}+bx^{4}+cx^{2}+\frac{1}{64}(a^2-4b)^2
\]
is an irreducible polynomial, where $a,b,c\in\mathbb{Q}$.  
Let $x_{1}, x_{2}, x_{3}, x_{4}$ be the roots of the polynomial $x^{4}+ax^{3}+bx^{2}+cx+(a^2-4b)^2/64$. Vieta's formulas yield 
\begin{equation}\label{eq418}
\sum_{i=1}^{4}x_{i}=-a,\;\;\;\;\sum_{1\leq i<j\leq 4}x_{i}x_{j}=b\;\;\text{ and }\;\; \prod_{i=1}^{4}x_{i} = \frac{1}{64}(a^2-4b)^2.
\end{equation}
Moreover, one can check (e.g., using SageMath~\cite{sagemath}) that the identity
\[
\prod_{\epsilon_{i}=\pm 1}(y_1+\epsilon_{2}y_2+\epsilon_{3}y_3+\epsilon_{4}y_4) = \left(\left(\sum_{i=1}^{4}y_{i}^2\right)^{2} - 4\sum_{1\leq i<j\leq 4}y_{i}^2y_{j}^2\right)^{2} - 64\prod_{i=1}^{4}y_{i}^2
\]
holds in the polynomial ring $\mathbb{Q}[y_1,y_2,y_3,y_4]$. Substituting $y_j=\sqrt{x_j}$, for $j=1,2,3,4$, into this identity, and using the expressions \eqref{eq418}, we obtain that
\begin{equation*}
\begin{split} 
&\prod_{\epsilon_{i}=\pm 1}(\sqrt{x_{1}}+\epsilon_{2}\sqrt{x_{2}}+\epsilon_{3}\sqrt{x_{3}}+\epsilon_{4}\sqrt{x_{4}})\\
=&\left(\left(\sum_{i=1}^{4}x_{i}\right)^{2} - 4\sum_{1\leq i<j\leq 4}x_{i}x_{j}\right)^{2} - 64\prod_{i=1}^{4}x_{i}\\
=&(a^2-4b)^2 - 64\cdot\frac{1}{64}(a^2-4b)^2=0.
\end{split}
\end{equation*}
Hence, for some $\epsilon_1,\epsilon_2,\epsilon_3\in\{\pm 1\}$ we have that $\sqrt{x_{1}}+\epsilon_{2}\sqrt{x_{2}}+\epsilon_{3}\sqrt{x_{3}}+\epsilon_{4}\sqrt{x_{4}}=0$. We can rewrite this as
\[
\sqrt{x_{1}}=-\epsilon_{2}\sqrt{x_{2}}-\epsilon_{3}\sqrt{x_{3}}-\epsilon_{4}\sqrt{x_{4}}.
\]
Moreover, $\pm\sqrt{x_{1}}, \pm\sqrt{x_{2}}, \pm\sqrt{x_{3}}, \pm\sqrt{x_{4}}$ are the roots of the polynomial $p(x)$. 
Finally, selecting $\alpha_1=\sqrt{x_{1}}$, $\alpha_2=-\epsilon_{2}\sqrt{x_{2}}$, $\alpha_3=-\epsilon_{3}\sqrt{x_{3}}$ and 
$\alpha_4=-\epsilon_{4}\sqrt{x_{4}}$, in view of the last equality, we obtain four distinct roots $\alpha_1,\alpha_2,\alpha_3,\alpha_4$ of $p(x)$ 
that satisfy the relation $\alpha_1=\alpha_2+\alpha_3+\alpha_4$. 
This completes the proof of Theorem~\ref{t1}.
\end{proof}

\section{Proof of Theorem~\ref{t2}}\label{prooft2}

\begin{proof}[Proof of Theorem~\ref{t2}]
Suppose that some four distinct algebraic conjugates of an algebraic number $\alpha$ of degree $8$ satisfy the relation 
\begin{equation}\label{eq49}
\alpha_{1}+\alpha_{2}+\alpha_{3}+\alpha_{4}=0,
\end{equation}
where $\alpha_1,\alpha_2,\dotsc,\alpha_8$ are the algebraic conjugates of $\alpha$. 
Throughout the proof of the theorem, we consider only the relations $\alpha_{i_1}+\alpha_{i_2}+\alpha_{i_3}+\alpha_{i_4}=0$, where 
$\alpha_{i_1},\alpha_{i_2},\alpha_{i_3},\alpha_{i_4}$ are four distinct algebraic conjugates. 
Two relations $\alpha_{i_1}+\alpha_{i_2}+\alpha_{i_3}+\alpha_{i_4}=0$ and $\alpha_{j_1}+\alpha_{j_2}+\alpha_{j_3}+\alpha_{j_4}=0$
are called distinct, if 
\[
\{\alpha_{i_1},\alpha_{i_2},\alpha_{i_3},\alpha_{i_4}\}\neq \{\alpha_{j_1},\alpha_{j_2},\alpha_{j_3},\alpha_{j_4}\}.
\] 
Let $N$ be the number of all possible distinct such relations, involving algebraic conjugates of $\alpha$. Lemma $\ref{intro3}$ implies that $tr(\alpha)=\alpha_{1}+\alpha_{2}+\dotsb+\alpha_{8}=0$. For any relation 
\begin{equation}\label{eq607}
\alpha_{i_1}+\alpha_{i_2}+\alpha_{i_3}+\alpha_{i_4}=0,
\end{equation} 
let 
\[
\{i_1',i_2',i_3',i_4'\}=\{1,2,\dotsc,8\}\setminus\{i_1,i_2,i_3,i_4\}.
\] 
Then, in view of $tr(\alpha)=0$, we have that  $\alpha_{i_1'}+\alpha_{i_2'}+\alpha_{i_3'}+\alpha_{i_4'}=0$. 
Note that the sets $\{\alpha_{i_1},\alpha_{i_2},\alpha_{i_3},\alpha_{i_4}\}$ and $\{\alpha_{i_1'},\alpha_{i_2'},\alpha_{i_3'},\alpha_{i_4'}\}$ are disjoint and every $\alpha_k$, $k\in\{1,2,\dotsc,8\}$, belongs to exactly one of these two sets. Therefore, the number $N$ is even, and every algebraic conjugate of $\alpha$ appears in exactly half of the relations. 
Moreover, the relation \eqref{eq49} implies that we have at least two distinct relations:
\begin{equation}\label{eq19}
\begin{split}
\alpha_{1} + \alpha_{2} + \alpha_{3} + \alpha_{4} &= 0,\\
\alpha_{5} + \alpha_{6} + \alpha_{7} + \alpha_{8} &= 0.
\end{split}
\end{equation}
Hence, $N\geq 2$. We will prove that $N=2, 4, 6,$ or $8$ and will consider each case separately.


First, we will prove several lemmas. 
\begin{lemma}\label{10}
If $\alpha_{i}+\alpha_{j}+\alpha_{k}+\alpha_{l}=0$ and $\alpha_{m}+\alpha_{n}+\alpha_{p}+\alpha_{r}=0$ are two distinct relations such that 
\begin{equation*}
\{\alpha_{i}, \alpha_{j}, \alpha_{k}, \alpha_{l}\}\cap\{\alpha_{m}, \alpha_{n}, \alpha_{p}, \alpha_{r}\}\neq\varnothing,
\end{equation*}
then $|\{\alpha_{i}, \alpha_{j}, \alpha_{k}, \alpha_{l}\}\cap\{\alpha_{m}, \alpha_{n}, \alpha_{p}, \alpha_{r}\}|=2$.
\end{lemma}
\begin{proof}
Denote $A:=\{\alpha_{i}, \alpha_{j}, \alpha_{k}, \alpha_{l}\}$ and $B:=\{\alpha_{m}, \alpha_{n}, \alpha_{p}, \alpha_{r}\}$. 
We have that $|A\cap B|<4$, since the two given relations are distinct. Hence, $|A\cap B|=1,2$ or 3. Assume that $|A\cap B|=1$. Without loss of generality, we can assume that $A\cap B=\{\alpha_i\}=\{\alpha_m\}$. 
Then, by adding the given two relations and subtracting $tr(\alpha)=0$, we obtain $\alpha_{i}=\alpha_{s}$, where $i\neq s$. A contradiction. On the other hand, assume that $|A\cap B|=3$. 
Without loss of generality, we can assume that $\alpha_{i}=\alpha_{m}, \alpha_{j}=\alpha_{n}, \alpha_{k}=\alpha_{p}$, and $\alpha_{l}\neq \alpha_{r}$. Then, by subtracting one relation from the other, we obtain $\alpha_{l}=\alpha_{r}$, which is impossible. Thus, the claim follows. 
\end{proof}

\begin{lemma}\label{11}
For any distinct  $i,j\in\{1,2,\dotsc,8\}$, the sum $\alpha_{i}+\alpha_{j}$ appears in at most three distinct relations of the form \eqref{eq607}. If $\alpha_{i}+\alpha_{j}$ appears in three distinct relations of the form \eqref{eq607}, then $\alpha_{i}=-\alpha_{j}$.
\end{lemma}
\begin{proof}
First, assume that $\alpha_{i}+\alpha_{j}$ appears in exactly three distinct relations of the form \eqref{eq607}. In view of Lemma $\ref{10}$, these three relations are of the following form:
\begin{equation}\label{eq60}
\begin{split}
\alpha_{i} + \alpha_{j} + \alpha_{k_{1}} + \alpha_{k_{2}} &= 0,\\
\alpha_{i} + \alpha_{j} + \alpha_{k_{3}} + \alpha_{k_{4}} &= 0,\\
\alpha_{i} + \alpha_{j} + \alpha_{k_{5}} + \alpha_{k_{6}} &= 0,
\end{split}
\end{equation}
where $i,j,k_{1},k_{2},\dots,k_{6}$ are distinct. By adding all these relations and subtracting $tr(\alpha)=0$, 
we derive $\alpha_{i}=-\alpha_{j}$.
 
Now, assume for a contradiction that $\alpha_{i}+\alpha_{j}$ appears in at least four distinct relations of the form \eqref{eq607}. 
In view of Lemma $\ref{10}$, we have the relations \eqref{eq60} and a fourth relation $\alpha_{i} + \alpha_{j} + \alpha_{k_{7}} + \alpha_{k_{8}} = 0$, where $i,j,k_{1},k_{2},\dots,k_{8}$ are distinct. But this is impossible, since $\deg(\alpha)=8$. 
\end{proof} 

\begin{lemma}\label{13}
$N\leq 8$.
\end{lemma}
\begin{proof}
Assume for a contradiction that $N\geq10$. 
Without loss of generality, we also assume both equations in \eqref{eq19}. Note that $\alpha_{1}$ must appear in at least five of these equations (it already appears once in \eqref{eq19}). In view of \eqref{eq19} and Lemma $\ref{10}$, for the other four equations (having $\alpha_1$), we must use the sums: $\alpha_{1}+\alpha_{2}$, $\alpha_{1}+\alpha_{3}$, or $\alpha_{1}+\alpha_{4}$. The Pigeonhole principle and Lemma $\ref{11}$ imply that one of these sums must appear exactly twice in these four equations. Without loss of generality, say that it is $\alpha_{1}+\alpha_{2}$:
\begin{equation*}
\begin{split}
\alpha_{1} + \alpha_{2} + \alpha_{i} + \alpha_{j} &= 0,\\
\alpha_{1} + \alpha_{2} + \alpha_{k} + \alpha_{l} &= 0,
\end{split}
\end{equation*}
where $\alpha_{i}, \alpha_{j}, \alpha_{k}, \alpha_{l}\in\{\alpha_{5}, \alpha_{6}, \alpha_{7}, \alpha_{8}\}$ are distinct. 
Then $\alpha_{1}=-\alpha_{2}$, by Lemma $\ref{11}$. 
Note that if $\alpha_{1}+\alpha_{3}$ (or $\alpha_{1}+\alpha_{4}$) appears twice in two other equations, 
then $\alpha_{1}=-\alpha_{3}$ (or respectively $\alpha_{1}=-\alpha_{4}$), which implies that $\alpha_{2}=\alpha_{3}$ (or respectively $\alpha_{2}=\alpha_{4}$). 
A contradiction. 
Thus, the only remaining possibility is the following:
\begin{equation*}
\begin{split}
\alpha_{1} + \alpha_{3} + \alpha_{m} + \alpha_{n} &= 0,\\
\alpha_{1} + \alpha_{4} + \alpha_{m} + \alpha_{p} &= 0,
\end{split}
\end{equation*}
where $\alpha_{m}, \alpha_{n}, \alpha_{p}\in\{\alpha_{5}, \alpha_{6}, \alpha_{7}, \alpha_{8}\}$ are all distinct. But then, $\alpha_{1}+\alpha_{m}$ appears in exactly three distinct equations. This means that $\alpha_{1}=-\alpha_{m}$, by Lemma $\ref{11}$. Thus, $\alpha_{2}=\alpha_{m}$, a contradiction.  
\end{proof}

\begin{lemma}\label{12}
If $N=8$, then $\alpha_{i}+\alpha_{j}$ for some $i\neq j$, must appear in exactly three distinct relations of the form \eqref{eq607}.  
\end{lemma}
\begin{proof}
Lemma $\ref{11}$ implies that $\alpha_{i}+\alpha_{j}$ cannot appear in more than three distinct relations of the form \eqref{eq607}. 
Thus, assume for a contradiction that any sum of two distinct conjugates of $\alpha$ appears in at most two distinct relations of the form \eqref{eq607}. 
Without loss of generality, we assume both relations in \eqref{eq19}. 
Since $N=8$, $\alpha_{1}$ must appear in exactly four of these relations (it already appears once in \eqref{eq19}). 
Thus, in view of our assumption and Lemma $\ref{10}$, for the other three equations, 
we must use the sums: $\alpha_{1}+\alpha_{2}$, $\alpha_{1}+\alpha_{3}$, or $\alpha_{1}+\alpha_{4}$ exactly once for each relation. 
Hence, the four distinct relations, involving $\alpha_1$, are
\begin{equation}\label{eq63}
\begin{split}
\alpha_{1} + \alpha_{2} + \alpha_{3} + \alpha_{4}              &= 0,\\
\alpha_{1} + \alpha_{2} + \alpha_{k_{1}} + \alpha_{k_{2}} &= 0,\\
\alpha_{1} + \alpha_{3} + \alpha_{k_{1}} + \alpha_{k_{3}} &= 0,\\
\alpha_{1} + \alpha_{4} + \alpha_{k_{2}} + \alpha_{k_{3}} &= 0,
\end{split}
\end{equation}
where $\alpha_{k_{1}}, \alpha_{k_{2}}, \alpha_{k_{3}}\in\{\alpha_{5}, \alpha_{6}, \alpha_{7}, \alpha_{8}\}$ are distinct. But then, by adding all the relations in \eqref{eq63}, we obtain 
\begin{equation*}
4\alpha_{1} + 2\alpha_{2} + 2\alpha_{3} + 2\alpha_{4} + 2\alpha_{k_{1}} + 2\alpha_{k_{2}} + 2\alpha_{k_{3}} = 0. 
\end{equation*}
Since $tr(\alpha)=0$, we derive $\alpha_{1}=\alpha_{k_{4}}$, where $\alpha_{k_{4}}\in\{\alpha_{5}, \alpha_{6}, \alpha_{7}, \alpha_{8}\}$. A contradiction. Thus, the claim follows.
\end{proof}

Recall that $N$ is an even positive integer and $N\geq 2$. 
By Lemma~\ref{13}, we have four cases to consider: $N=2,4,6$ or 8. 

\vspace{0,3cm}

\underline{Case $N=8$}. 
Lemma~\ref{11} together with Lemma~\ref{12} imply that  $\alpha_{i}=-\alpha_{j}$ for some  $i\neq j$. 
Hence, the minimal polynomial $p(x)$ of $\al$ is of the form $p_4(x)$ (an even polynomial), given in of Theorem~\ref{t2}. 
Conversely, let $p(x)=p_4(x)$ be an irreducible polynomial, given in Theorem~\ref{t2}.  
Since $p(x)$ is even, we can easily select four distinct roots of $p(x)$ that sum to zero (note that the 
roots of an even polynomial $p(x)$ come in pairs $(\alpha,-\alpha)$, so that selecting two distinct 
pairs $(\alpha,-\alpha)$ and $(\alpha',-\alpha')$, where $\alpha\neq\pm\alpha'$, we obtain four distinct roots of $p(x)$ that sum to zero). 

\begin{remark}\label{evenpol}
From now on, we assume  that the minimal polynomial of $\alpha$ is not an even polynomial, since we already proved that in such a case, there exists a relation of 
the form \eqref{eq607} (see Case $N=8$).    
Note that this assumption is equivalent to $\alpha_i\neq -\alpha_j$ for any two distinct $i$ and $j$.
\end{remark}

\vspace{0,3cm}

\underline{Case $N=2$}. Without loss of generality, we assume the relations in \eqref{eq19}. 
Since $N=2$, there are no other relations of the form \eqref{eq607} apart from those in \eqref{eq19}. 
Take two polynomials 
\begin{equation*}
p_{1}(x) := (x-\alpha_{1})(x-\alpha_{2})(x-\alpha_{3})(x-\alpha_{4}) = x^{4} + ux^{2} + vx + w,
\end{equation*}
\begin{equation*}
p_{2}(x) := (x-\alpha_{5})(x-\alpha_{6})(x-\alpha_{7})(x-\alpha_{8}) = x^{4} + u'x^{2} + v'x + w'.
\end{equation*}
Their product is the minimal polynomial of $\alpha$:
\begin{equation}\label{eq745}
p(x) = (x^{4} + ux^{2} + vx + w)(x^{4} + u'x^{2} + v'x + w')
\end{equation}
that expands to
\begin{equation}\label{eq22}
\begin{split}
p(x) &= x^{8} + (u+u')x^{6} + (v+v')x^{5} + (uu'+w+w')x^{4} + (uv'+vu')x^{3}\\
&+ (uw'+vv'+wu')x^{2} + (vw'+wv')x + ww'.                   
\end{split}
\end{equation}

Denote $\mathcal{S}_{1}:=\{\alpha_{1}, \alpha_{2}, \alpha_{3}, \alpha_{4}\}$ and $\mathcal{S}_{2}:=\{\alpha_{5}, \alpha_{6}, \alpha_{7}, \alpha_{8}\}$. 
Let $G$ be the Galois group of the normal closure of $\mathbb{Q}(\alpha_{1})$ over $\mathbb{Q}$. The group $G$ corresponds to some transitive subgroup of the full symmetric group $S_{8}$. Take any automorphism $\phi\in G$. Note that either $\phi(\mathcal{S}_{1})=\mathcal{S}_{1}$ or $\phi(\mathcal{S}_{1})=\mathcal{S}_{2}$. Since otherwise, by acting on $\alpha_{1} + \alpha_{2} + \alpha_{3} + \alpha_{4} = 0$ with $\phi$, we would obtain a relation $\alpha_{i_{1}} + \alpha_{i_{2}} + \alpha_{i_{3}} + \alpha_{i_{4}} = 0$ that does not exist. Therefore, 
\begin{equation}\label{eq768}
\phi(\mathcal{S}_{1})=\mathcal{S}_{1}\; \;\text{and}\;\; \phi(\mathcal{S}_{2})=\mathcal{S}_{2}\;\;\; \text{or}\; \;\;
\phi(\mathcal{S}_{1})=\mathcal{S}_{2}\;\; \text{and}\;\; \phi(\mathcal{S}_{2})=\mathcal{S}_{1}.
\end{equation} 
Let $s_4(x_1,x_2,x_3,x_4) = x_1x_2x_3x_4$ be an elementary symmetric polynomial. Note that $w=s_4(\alpha_{1}, \alpha_{2}, \alpha_{3}, \alpha_{4})$ and $w'=s_4(\alpha_{5}, \alpha_{6}, \alpha_{7}, \alpha_{8})$. For any automorphism $\phi\in G$, in view of \eqref{eq768}, we have that 
$\phi(w)=w$ or $\phi(w)=w'$. Also, selecting $\phi\in G$ which sends $\alpha_1$ to $\alpha_5$, we obtain $\phi(w)=w'$. 
So that $w$ and $w'$ are algebraic conjugates of degree at most 2. Similarly, we get that $v$ and $v'$ are algebraic conjugates of degree at most 2, and $u$ and $u'$ are algebraic conjugates of degree at most 2. On the other hand, numbers $u, v$ and $w$ can't all be rational, since then the polynomial $p_1(x)$ would have rational coefficients,  contradicting the irreducibility of $p(x)$. 
Hence, at least one of the numbers $u, v$ and $w$ is a quadratic algebraic number. 

We will prove that all the numbers $u,v$ and $w$ belong to the same quadratic field. For this, we need the following.

\begin{lemma}\label{le779}
Suppose that $y,y',z$ and $z'$ are complex numbers, such that all three numbers $y+y'$, $z+z'$ and $yz'+y'z$ are rational. 
If $y\notin\mathbb{Q}$, then $z\in\mathbb{Q}(y)$.
\end{lemma}

\begin{proof}
Let $y+y'=r_1\in\mathbb{Q}$ and $z+z'=r_2\in\mathbb{Q}$. 
Plugging the expressions $y'=r_1-y$ and $z'=r_2-z$ into $yz'+y'z=r_3\in\mathbb{Q}$, we obtain 
\[
r_3=yz'+y'z = y(r_2-z) + (r_1-y)z = r_2y+z(r_1-2y).
\]
Note that $r_1-2y\neq 0$, since $y\notin\mathbb{Q}$. Therefore $z=(r_3-r_2y)/(r_1-2y)\in\mathbb{Q}(y)$.
\end{proof}

Considering the coefficients of $x$, $x^2$ and $x^3$ in \eqref{eq22} and taking into account that $vv'$ is a rational number, we obtain that all three numbers $uv'+vu'$,  $uw'+wu'$ and $vw'+wv'$ are rational. 
Recall that at least one of the numbers $u, v$ and $w$ is a quadratic algebraic number. If $u\notin\mathbb{Q}$, then, applying 
Lemma~\ref{le779} to rational numbers $uv'+vu'$ and  $uw'+wu'$, we get that both $v$ and $w$ belong to $\mathbb{Q}(u)$.  
Similarly, we obtain that if $v\notin\mathbb{Q}$, then $u,w\in\mathbb{Q}(v)$, and if  
$w\notin\mathbb{Q}$, then $u,v\in\mathbb{Q}(w)$. This completes the proof that all the numbers $u,v$ and $w$ belong to the same quadratic field.


Let $t\neq 1$ be a square-free integer such that $u,v,w\in\mathbb{Q}(\sqrt{t})$.  
Since both numbers in any pair $(u,u')$, $(v,v')$ and $(w,w')$ are algebraic conjugates, we can express these numbers as follows:
\begin{align*} 
u &=  a+b\sqrt{t}, & v &=  c+d\sqrt{t}, &  w &=  e+f\sqrt{t},\\ 
u' &=  a-b\sqrt{t}, &  v' &=  c-d\sqrt{t}, & w' &=  e-f\sqrt{t},
\end{align*}
where $a,b,c,d,e,f\in\mathbb{Q}$. Then, in view of \eqref{eq745}, 
\begin{align}\label{eq809}
p(x) = &(x^{4} + ux^{2} + vx + w)(x^{4} + u'x^{2} + v'x + w')\nonumber\\
       = &\left( x^{4} + (a+b\sqrt{t})x^{2} + (c+d\sqrt{t})x + e+f\sqrt{t} \right)\cdot\\
          & \left( x^{4} + (a-b\sqrt{t})x^{2} + (c-d\sqrt{t})x + e-f\sqrt{t} \right).\nonumber
\end{align} 
By expanding this product, we obtain that $p(x)$ is of the form $p_1(x)$, given in  Theorem~\ref{t2}.

Conversely, let $p(x)=p_1(x)$ be an irreducible polynomial, given in Theorem~\ref{t2}. 
Then $p(x)$ factors as in \eqref{eq809} and the four roots of the first factor sum to zero.

\vspace{0,3cm}

\underline{Case $N=4$}. 
Without loss of generality, we assume the relations 
\begin{equation}\label{eq8722}
\begin{split}
\alpha_{1} + \alpha_{2} + \alpha_{3} + \alpha_{4} &= 0,\\
\alpha_{5} + \alpha_{6} + \alpha_{7} + \alpha_{8} &= 0.
\end{split}
\end{equation}
Since $N=4$, there exists a relation $\alpha_{i} + \alpha_{j} + \alpha_{k} + \alpha_{l} = 0$  that is distinct from these two relations. 
By Lemma~\ref{10}, $|\{ \alpha_{i}, \alpha_{j}, \alpha_{k}, \alpha_{l}\}\cap \{ \alpha_{1}, \alpha_{2}, \alpha_{3}, \alpha_{4}\}|=2$  and 
$|\{ \alpha_{i}, \alpha_{j}, \alpha_{k}, \alpha_{l}\}\cap \{ \alpha_{5}, \alpha_{6}, \alpha_{7}, \alpha_{8}\}|=2$. 
Without loss of generality, we can assume that $i=1$, $j=2$, $k=5$ and $l=6$. 
So we have $\alpha_{1} + \alpha_{2} + \alpha_{5} + \alpha_{6} = 0$. From this relation and the fact that $tr(\alpha)=0$ we 
derive one more relation $\alpha_{3} + \alpha_{4} + \alpha_{7} + \alpha_{8} = 0$. Therefore, we have four distinct relations:
\begin{equation}\label{eq884}
\begin{split}
\alpha_{1} + \alpha_{2} + \alpha_{3} + \alpha_{4} &= 0,\\
\alpha_{5} + \alpha_{6} + \alpha_{7} + \alpha_{8} &= 0,\\
\alpha_{1} + \alpha_{2} + \alpha_{5} + \alpha_{6} &= 0,\\
\alpha_{3} + \alpha_{4} + \alpha_{7} + \alpha_{8} &= 0.
\end{split}
\end{equation}
Since $N=4$, there are no other relations involving four distinct conjugates of $\alpha$. These relations immediately imply the following equalities:
\begin{equation}\label{eq24}
\begin{split}
\alpha_{1} + \alpha_{2} &= \alpha_{7} + \alpha_{8} = \beta,\\
\alpha_{3} + \alpha_{4} &= \alpha_{5} + \alpha_{6} = -\beta.
\end{split}
\end{equation}
We will prove that $\beta$ and $-\beta$ are quadratic conjugates. 
Note that the Galois group $G$ permutes the relations in  \eqref{eq884} and for any 
two distinct such relations $\alpha_{i_1}+\alpha_{i_2}+\alpha_{i_3}+\alpha_{i_4}=0$  and 
$\alpha_{j_1}+\alpha_{j_2}+\alpha_{j_3}+\alpha_{j_4}=0$ we have that the intersection
\[
\{\alpha_{i_1},\alpha_{i_2},\alpha_{i_3},\alpha_{i_4}\}  \cap \{\alpha_{j_1},\alpha_{j_2},\alpha_{j_3},\alpha_{j_4}\}  
\]
equals $\emptyset$, $\{\alpha_1,\alpha_2\}$, $\{\alpha_3,\alpha_4\}$, $\{\alpha_5,\alpha_6\}$ or $\{\alpha_7,\alpha_8\}$. 
Also, any automorphism $\phi$ in $G$ maps two distinct relations in  \eqref{eq884} to distinct relations (since $\phi$ is injective and any two distinct relations contain in total at least 6 distinct conjugates of $\alpha_1$). 
Moreover, any conjugate $\alpha_i$ in the non-empty such intersection uniquely determines the intersection itself. 
Thus, if $\phi(\alpha_1)=\alpha_3$, then $\phi(\alpha_2)=\alpha_4$. 
Similarly, any automorphism in $G$ maps $\{\alpha_{1}, \alpha_{2}\}$ to $\{\alpha_{1}, \alpha_{2}\}, \{\alpha_{3}, \alpha_{4}\}, \{\alpha_{5}, \alpha_{6}\}$ or $\{\alpha_{7}, \alpha_{8}\}$. 
This, in view of \eqref{eq24}, implies that $\beta$ and $-\beta$ are algebraic conjugates and $\deg(\beta)\leq 2$. 
Now, we show that $\beta\notin\mathbb{Q}$. Indeed, if $\beta\in\mathbb{Q}$, then $\beta=-\beta=0$. 
But in that case, we would derive a new relation $\alpha_{1} + \alpha_{2} + \alpha_{7} + \alpha_{8} = 0$. 
A contradiction. 
Hence, $\beta$ and $-\beta$ are quadratic conjugates. 
Denote $\beta=a\sqrt{t}$, where $0\neq a\in\mathbb{Q}$ and $t\neq 1$ is a square-free integer. 
Now the minimal polynomial of $\alpha$ can be expressed as follows:
\begin{equation}\label{eq25}
\begin{split}
p(x) =\ &(x^{2} - a\sqrt{t}x + \alpha_{1}\alpha_{2})(x^{2} - a\sqrt{t}x + \alpha_{7}\alpha_{8})\cdot\\
&(x^{2} + a\sqrt{t}x + \alpha_{3}\alpha_{4})(x^{2} + a\sqrt{t}x + \alpha_{5}\alpha_{6}).
\end{split}
\end{equation}
The idea here is to pair the factors that include $\alpha_{1}\alpha_{2}$ and $\alpha_{7}\alpha_{8}$. Similarly, we pair factors that include $\alpha_{3}\alpha_{4}$ and $\alpha_{5}\alpha_{6}$:
\begin{equation}\label{eq26}
\begin{split}
p(x) =\ &\left(x^{4} - 2a\sqrt{t}x^{3} + (\alpha_{1}\alpha_{2} + \alpha_{7}\alpha_{8} + a^{2}t)x^{2}\right.\\
&\left. - a\sqrt{t}(\alpha_{1}\alpha_{2} + \alpha_{7}\alpha_{8})x + \alpha_{1}\alpha_{2}\alpha_{7}\alpha_{8}\right)\cdot\\
&\left(x^{4} + 2a\sqrt{t}x^{3} + (\alpha_{3}\alpha_{4} + \alpha_{5}\alpha_{6} + a^{2}t)x^{2}\right.\\
&\left. + a\sqrt{t}(\alpha_{3}\alpha_{4} + \alpha_{5}\alpha_{6})x + \alpha_{3}\alpha_{4}\alpha_{5}\alpha_{6}\right).
\end{split}
\end{equation}
We claim that $\alpha_{1}\alpha_{2} + \alpha_{7}\alpha_{8}$, $\alpha_{3}\alpha_{4} + \alpha_{5}\alpha_{6}$, $\alpha_{1}\alpha_{2}\alpha_{7}\alpha_{8}$ and $\alpha_{3}\alpha_{4}\alpha_{5}\alpha_{6}$ all lie in the same quadratic field $K=\mathbb{Q}(\sqrt{t})$. Indeed, let 
\[
p(x)=x^8+c_6x^6+c_5x^5+c_4x^4+c_3x^3+c_2x^2+c_1x+c_0\in\mathbb{Q}[x].
\] 
(Note that $c_7=0$, since $tr(\alpha_1)=0$.)
By expanding the expression \eqref{eq25}, we obtain:
\begin{equation*}
\begin{split}
c_{1} &= a\sqrt{t}\left(\alpha_{1}\alpha_{2}\alpha_{7}\alpha_{8}(\alpha_{3}\alpha_{4} + \alpha_{5}\alpha_{6}) - \alpha_{3}\alpha_{4}\alpha_{5}\alpha_{6}(\alpha_{1}\alpha_{2} + \alpha_{7}\alpha_{8})\right),\\
c_{3} &= a^{3}t\sqrt{t}(\alpha_{3}\alpha_{4} + \alpha_{5}\alpha_{6} - \alpha_{1}\alpha_{2} - \alpha_{7}\alpha_{8}) + 2a\sqrt{t}(\alpha_{1}\alpha_{2}\alpha_{7}\alpha_{8} - \alpha_{3}\alpha_{4}\alpha_{5}\alpha_{6}),\\
c_{4} &= a^{4}t^{2} - a^{2}t(\alpha_{1}\alpha_{2} + \alpha_{3}\alpha_{4} + \alpha_{5}\alpha_{6} + \alpha_{7}\alpha_{8})\\
&+ (\alpha_{1}\alpha_{2} + \alpha_{7}\alpha_{8})(\alpha_{3}\alpha_{4} + \alpha_{5}\alpha_{6}) + \alpha_{1}\alpha_{2}\alpha_{7}\alpha_{8} + \alpha_{3}\alpha_{4}\alpha_{5}\alpha_{6},\\
c_{5} &= a\sqrt{t}(\alpha_{1}\alpha_{2} + \alpha_{7}\alpha_{8} - \alpha_{3}\alpha_{4} - \alpha_{5}\alpha_{6}),\\
c_{6} &= \alpha_{1}\alpha_{2} + \alpha_{7}\alpha_{8} + \alpha_{3}\alpha_{4} + \alpha_{5}\alpha_{6} - 2a^{2}t.
\end{split}
\end{equation*}
The coefficients $c_{5}$ and $c_{6}$ imply 
\begin{equation*}
\begin{split}
\alpha_{1}\alpha_{2} + \alpha_{7}\alpha_{8} &= \frac{c_{6}}{2} + a^{2}t + \frac{c_{5}}{2at}\sqrt{t} = b+c\sqrt{t},\\
\alpha_{3}\alpha_{4} + \alpha_{5}\alpha_{6} &= \frac{c_{6}}{2} + a^{2}t - \frac{c_{5}}{2at}\sqrt{t} = b-c\sqrt{t},
\end{split}
\end{equation*}
where $b,c\in\mathbb{Q}$. Thus, the numbers $\alpha_{1}\alpha_{2} + \alpha_{7}\alpha_{8}$ and $\alpha_{3}\alpha_{4} + \alpha_{5}\alpha_{6}$ 
lie in a quadratic field $K$. Similarly, the coefficients $c_{3}$ and $c_{5}$ imply
\begin{equation}\label{eq27}
\alpha_{1}\alpha_{2}\alpha_{7}\alpha_{8} - \alpha_{3}\alpha_{4}\alpha_{5}\alpha_{6} = \Big(\frac{c_{3}}{2at}+\frac{ac_{5}}{2}\Big)\sqrt{t} = 2e\sqrt{t},
\end{equation}
where $e\in\mathbb{Q}$. Note that $c$ and $e$ are not both zero, since otherwise $c_{1}=c_{3}=c_{5}=0$ and we obtain an even polynomial. This immediately implies $\alpha_{i}=-\alpha_{j}$, which means that we have at least six equations of the form $\alpha_{i_{1}}+\alpha_{i_{2}}+\alpha_{i_{3}}+\alpha_{i_{4}}=0$. A contradiction. By using the above results, we obtain 
\begin{equation*}
c_{4} = a^{4}t^{2} - 2a^{2}bt + b^{2} - c^{2}t + \alpha_{1}\alpha_{2}\alpha_{7}\alpha_{8} + \alpha_{3}\alpha_{4}\alpha_{5}\alpha_{6}
\end{equation*}
This implies 
\begin{equation}\label{eq28}
\alpha_{1}\alpha_{2}\alpha_{7}\alpha_{8} + \alpha_{3}\alpha_{4}\alpha_{5}\alpha_{6} = 2d, 
\end{equation}
where $d\in\mathbb{Q}$. Equations \eqref{eq27} and \eqref{eq28} imply that
\begin{equation*}
\begin{split}
\alpha_{1}\alpha_{2}\alpha_{7}\alpha_{8} &= d+e\sqrt{t},\\
\alpha_{3}\alpha_{4}\alpha_{5}\alpha_{6} &= d-e\sqrt{t}.
\end{split}
\end{equation*}
Hence, the numbers $\alpha_{1}\alpha_{2}\alpha_{7}\alpha_{8}$ and $\alpha_{3}\alpha_{4}\alpha_{5}\alpha_{6}$ 
lie in a quadratic field $K$ and the expression of $p(x)$ in \eqref{eq26} becomes
\begin{equation}\label{eq894}
\begin{split}
p(x) =\ &(x^{4} - 2a\sqrt{t}x^{3} + (b+c\sqrt{t} + a^{2}t)x^{2} - a\sqrt{t}(b+c\sqrt{t})x + d+e\sqrt{t})\cdot\\
&(x^{4} + 2a\sqrt{t}x^{3} + (b-c\sqrt{t} + a^{2}t)x^{2} + a\sqrt{t}(b-c\sqrt{t})x + d-e\sqrt{t}).
\end{split}
\end{equation}
By expanding this product, we obtain that $p(x)$ is of the form $p_2(x)$, given in  Theorem~\ref{t2}.

Conversely, let $p(x)=p_2(x)$ be an irreducible polynomial, given in Theorem~\ref{t2}. 
Then $p(x)$ factors as in \eqref{eq894}. Let $u,u',v$ and $v'$ be complex numbers, defined as follows:
\begin{align*} 
u+u' &=  b+c\sqrt{t}, \\ 
uu' &=  d+e\sqrt{t},\\
v+v' &=  b-c\sqrt{t}, \\ 
vv' &=  d-e\sqrt{t}.
\end{align*}
One can check (e.g., using SageMath~\cite{sagemath}) that the product $(x^2-a\sqrt{t}x+u)(x^2-a\sqrt{t}x+u')$ equals the first factor in \eqref{eq894} and the product  $(x^2+a\sqrt{t}x+v)(x^2+a\sqrt{t}x+v')$ equals the second factor in \eqref{eq894}. Hence, 
\[
p(x) = (x^2-a\sqrt{t}x+u)(x^2-a\sqrt{t}x+u')(x^2+a\sqrt{t}x+v)(x^2+a\sqrt{t}x+v').
\]
Let $\alpha_1$ and $\alpha_2$ be the roots of the polynomial $x^2-a\sqrt{t}x+u$. Then $\alpha_1+\alpha_2=a\sqrt{t}$. 
Similarly, let  $\alpha_3$ and $\alpha_4$ be the roots of the polynomial $x^2+a\sqrt{t}x+v$. Then $\alpha_3+\alpha_4=-a\sqrt{t}$. 
Therefore $\alpha_1+\alpha_2+\alpha_3+\alpha_4=0$.

\vspace{0,3cm}

\underline{Case $N=6$}. 
Without loss of generality, we assume the relations 
\begin{equation}\label{eq1019}
\begin{split}
\alpha_{1} + \alpha_{2} + \alpha_{3} + \alpha_{4} &= 0,\\
\alpha_{5} + \alpha_{6} + \alpha_{7} + \alpha_{8} &= 0,\\
\alpha_{1} + \alpha_{2} + \alpha_{5} + \alpha_{6} &= 0,\\
\alpha_{3} + \alpha_{4} + \alpha_{7} + \alpha_{8} &= 0.
\end{split}
\end{equation}
Recall that every algebraic conjugate of $\alpha$ appears exactly in half of the relations. 
Since $N=6$ and $\alpha_1$ appears in exaclty two relations in \eqref{eq1019}, 
there exists a relation $\alpha_{1} + \alpha_{j} + \alpha_{k} + \alpha_{l} = 0$  that is distinct from the four relations in \eqref{eq1019}. 
By Lemma~\ref{10}, $|\{ \alpha_{1}, \alpha_{j}, \alpha_{k}, \alpha_{l}\}\cap \{ \alpha_{1}, \alpha_{2}, \alpha_{3}, \alpha_{4}\}|=2$  and 
$|\{ \alpha_{1}, \alpha_{j}, \alpha_{k}, \alpha_{l}\}\cap \{ \alpha_{5}, \alpha_{6}, \alpha_{7}, \alpha_{8}\}|=2$. 
Without loss of generality, we can assume that $j\in\{2,3,4\}$ and $k,l\in\{5,6,7,8\}$.  
Note that $j\neq 2$, since otherwise $\alpha_1+\alpha_2$ would appear in three distinct relations and, 
in view of Lemma~\ref{11}, the minimal polynomial of $\alpha$ would be an even polynomial (see Remark~\ref{evenpol}). Hence, $j\in\{3,4\}$. Without loss of generality, assume that $j=3$. 
Moreover, by Lemma~\ref{10}, $|\{ \alpha_{1}, \alpha_{3}, \alpha_{k}, \alpha_{l}\}\cap \{ \alpha_{1}, \alpha_{2}, \alpha_{5}, \alpha_{6}\}|=2$  and 
$|\{ \alpha_{1}, \alpha_{3}, \alpha_{k}, \alpha_{l}\}\cap \{ \alpha_{3}, \alpha_{4}, \alpha_{7}, \alpha_{8}\}|=2$. 
Thus, $|\{k,l\}\cap\{5,6\}|=1$ and $|\{k,l\}\cap\{7,8\}|=1$. 
Without loss of generality, we assume that $k=5$ and $l=7$. 
We obtained the relation $\alpha_{1} + \alpha_{3} + \alpha_{5} + \alpha_{7} = 0$.  
From this and the fact that $tr(\alpha)=0$ we 
derive one more relation $\alpha_{2} + \alpha_{4} + \alpha_{6} + \alpha_{8} = 0$. 
Therefore, we have six distinct relations:
\begin{equation}\label{eq1053}
\begin{split}
\alpha_{1} + \alpha_{2} + \alpha_{3} + \alpha_{4} &= 0,\\
\alpha_{5} + \alpha_{6} + \alpha_{7} + \alpha_{8} &= 0,\\
\alpha_{1} + \alpha_{2} + \alpha_{5} + \alpha_{6} &= 0,\\
\alpha_{3} + \alpha_{4} + \alpha_{7} + \alpha_{8} &= 0,\\
\alpha_{1} + \alpha_{3} + \alpha_{5} + \alpha_{7} &= 0,\\
\alpha_{2} + \alpha_{4} + \alpha_{6} + \alpha_{8} &= 0.
\end{split}
\end{equation}
Since $N=6$, there are no other relations involving four distinct conjugates of $\alpha$.

A computation with SageMath \cite{sagemath} shows that there are exactly 
16 transitive subgroups of the symmetric group $S_8$ which preserve the linear relations  \eqref{eq1053} (the code of this and further computations is provided in Listing~\ref{lst:2} in the Appendix). 
These are given in Table~\ref{table:tsS8pN6}.  
\begin{table}[h!]
\centering
\begin{tabular}{ |c|c|c| } 
\hline
Group $G$ & $\# G$ & Generators of $G$\\
\hline
$C_2\times C_2\times C_2$ & 8 & $(1\,2)(3\,4)(5\,6)(7\,8)$, $(1\,7)(2\,8)(3\,5)(4\,6)$,  \\
                                            &    & $(1\,4)(2\,3)(5\,8)(6\,7)$\\
 $C_4\times C_2$ & 8 & $(1\,4\,7\,6)(2\,3\,8\,5)$, $(1\,3\,7\,5)(2\,4\,8\,6)$ \\
 $C_4\times C_2$ & 8 & $(1\,3\,4\,2)(5\,7\,8\,6)$, $(1\,7\,4\,6)(2\,5\,3\,8)$ \\
 $C_4\times C_2$ & 8 & $(1\,4\,6\,7)(2\,8\,5\,3)$, $(1\,2\,6\,5)(3\,4\,8\,7)$ \\
 $D_4$   & 8  & $(1\,2)(3\,4)(5\,6)(7\,8)$, $(1\,6\,4\,7)(2\,8\,3\,5)$\\
 $D_4$   & 8  & $(1\,8)(2\,4)(3\,6)(5\,7)$, $(1\,7\,6\,4)(2\,3\,5\,8)$ \\
  $D_4$   & 8  & $(1\,2)(3\,6)(4\,5)(7\,8)$, $(1\,4\,7\,6)(2\,3\,8\,5)$ \\
  $D_4$   & 8  &  $(1\,4)(2\,3)(5\,8)(6\,7)$, $(1\,5\,7\,3)(2\,6\,8\,4)$\\
  $D_4$   & 8  & $(1\,5)(2\,7)(3\,6)(4\,8)$, $(1\,3\,4\,2)(5\,7\,8\,6)$ \\
  $D_4$   & 8  & $(1\,7)(2\,8)(3\,5)(4\,6)$, $(1\,2\,6\,5)(3\,4\,8\,7)$ \\
 $C_2\times D_4$  & 16  & $(1\,8)(2\,7)(3\,4)(5\,6)$, $(1\,5)(2\,6)(3\,7)(4\,8)$, $(1\,7)(2\,8)$\\
  $C_2\times D_4$  & 16  & $(1\,5)(2\,7)(3\,6)(4\,8)$, $(1\,6)(2\,5)(3\,8)(4\,7)$, $(1\,4)(5\,8)$ \\
  $C_2\times D_4$  & 16  &  $(1\,8)(2\,4)(3\,6)(5\,7)$, $(1\,4)(2\,3)(5\,8)(6\,7)$, $(2\,5)(4\,7)$\\
  $C_2\times A_4$  & 16  & $(1\,8)(2\,6\,5\,7\,3\,4)$, $(1\,7)(2\,8)(3\,5)(4\,6)$\\
  $S_4$ & 24  & $(1\,8)(2\,7)(3\,4)(5\,6)$, $(1\,2\,4\,3)(5\,6\,8\,7)$\\
$C_2\times S_4$   & 48  & $(2\,3\,5)(4\,7\,6)$, $(1\,4\,7\,6)(2\,3\,8\,5)$, $(1\,3\,7\,5)(2\,4\,8\,6)$\\     
\hline
\end{tabular}
\vspace{0,2cm}
\caption{All the transitive subgroups of $S_8$ that preserve all the relations in \eqref{eq1053}. }
\label{table:tsS8pN6}
\end{table}
One can easily check that for every such subgroup, in view of the relations \eqref{eq1053}, the orbit of $\alpha_1-\alpha_8$ is $\{\alpha_1-\alpha_8, \alpha_8-\alpha_1\}$ and the orbit of $\alpha_1+\alpha_8$ is 
$\{\alpha_1+\alpha_8,\alpha_2+\alpha_7,\alpha_3+\alpha_6,\alpha_4+\alpha_5\}$. 
Note that $\alpha_1-\alpha_8\neq \alpha_8-\alpha_1$, since otherwise $\alpha_1=\alpha_8$, which is impossible.  Hence, $2\gamma:=\alpha_1-\alpha_8$ is a quadratic 
algebraic number with zero trace and $2\delta:=\alpha_1+\alpha_8$ is of degree at most 4.  
Then $\alpha_1=\gamma+\delta$ and 
\[
8=\deg(\alpha_1) =\deg\left( \gamma+\delta\right) \leq \deg\left( \gamma\right)\cdot \deg\left( \delta\right) = 2\cdot 4=8.
\]  
Thus $\delta$ is of degree 4 and $tr(2\delta) =(\alpha_1+\alpha_8)+(\alpha_2+\alpha_7)+(\alpha_3+\alpha_6)+(\alpha_4+\alpha_5)=tr(\alpha_1)=0$. So that the trace of $\delta$ equals 0. Let $\gamma=\gamma_1$ and $\gamma_2$ be the algebraic conjugates of $\gamma$ and 
$\delta=\delta_1, \delta_2,\delta_3$ and $\delta_4$ be the algebraic conjugates of $\delta$. 
Since $\alpha_1$ has degree 8 and equals the sum of a quadratic algebraic number $\gamma$ and a quartic algebraic number $\delta$, we obtain that $\gamma_i+\delta_j$, for $i=1,2$ and $j=1,2,3,4$, are the algebraic conjugates of $\alpha_1$.
  
The minimal polynomial of $\gamma$ is of the form $x^{2}-a$ and the minimal polynomial of $\delta$ is of the form $R(x)=x^{4}+bx^{2}+cx+d$, where $a,b,c,d\in\mathbb{Q}$ (recall that $tr(\gamma)=tr(\delta)=0$). 
Now, the minimal polynomial $p(x)$ of $\alpha_1$ can be expressed as
\begin{equation*}
\begin{split}
p(x) &= \prod_{\substack{i=1,2\\j=1,2,3,4}} (x-\gamma_i-\delta_j)\\
&=\prod_{i=1,2} (x-\gamma_i-\delta_1)(x-\gamma_i-\delta_2)(x-\gamma_i-\delta_3)(x-\gamma_i-\delta_4)\\
&= R(x-\gamma_1)R(x-\gamma_2) = R(x-\sqrt{a})R(x+\sqrt{a}).
\end{split}
\end{equation*}
By expanding this product, we obtain that $p(x)$ is of the form $p_3(x)$, given in  Theorem~\ref{t2}.

Conversely, let $p(x)=p_3(x)$ be an irreducible polynomial, given in Theorem~\ref{t2}. 
Then one can check (e.g., using SageMath~\cite{sagemath}) that $p(x)$ factors as $p(x) = R(x-\sqrt{a})R(x+\sqrt{a})$, where $R(x)=x^{4}+bx^{2}+cx+d$. 
Let $\delta_1, \delta_2,\delta_3$ and $\delta_4$ be the roots of $R(x)$. 
Then $\pm\sqrt{a}+\delta_j$, for $j=1,2,3,4$, are the roots of $p(x)$.  
Selecting $\alpha_1:=\sqrt{a}+\delta_1$, $\alpha_2:=\sqrt{a}+\delta_2$, $\alpha_3:=-\sqrt{a}+\delta_3$ and 
$\alpha_4:=-\sqrt{a}+\delta_4$, we obtain four distinct roots of $p(x)$ that sum to zero: 
\[
\alpha_1+\alpha_2+\alpha_3+\alpha_4=\delta_1 + \delta_2+\delta_3+\delta_4=0.
\]
\end{proof}


\printbibliography

\appendix
\section{Source code of SageMath calculations}

\begin{lstlisting}[language=Python,basicstyle=\ttfamily\scriptsize,caption=SageMath code used in the proof of Theorem~\ref{t1}.,label={lst:1}]
# Generates the list of all the transitive subgroups of S_8.
G = SymmetricGroup(8)
all_subgroups = G.subgroups()
transitive_subgroups = [H for H in all_subgroups if H.is_transitive()]

# For every transitive subgroup H, calculate the following:
# (Code segment C1) The list Relations, which corresponds to the orbit of the relation $\alpha_1=\alpha_2+\alpha_3+\alpha_4$ under H. Each list [i,j,k,l], i,j,k,l in {1,2,...,8}, j<k<l, in Relations corresponds to the relation $\alpha_i=\alpha_j+\alpha_k+\alpha_l$.
# (Code segment C2) Creates matrices R and R_i, which correspond to matrices $\mathcal{R}$ and $\mathcal{R}_i$, respectively, described in the proof of Theorem (*@\ref{t1}@*). 
# (Code segment C3) Checks if rank(R) = rank(R_i) for some i. In such case we obtain that $\alpha_1=-\alpha_i$. 
# (Code segment C4) If rank(R) = rank(R_i) are distinct for all i, then the list Suitable is appended with the subgroup H.
# (Code segment C5) If rank(R) = rank(R_i) for two distinct values of i (i=j and i=k), then we obtain that $\alpha_j=\alpha_k$ (see the proof of Theorem (*@\ref{t1}@*)), which is impossible. If rank(R) = rank(R_i) for exactly one value of i, then the list Galois_group_candidates is appended with the subgroup H.
Suitable = []
Galois_group_candidates = []
for H in transitive_subgroups:
    Relations = []
    for f in H.list():
        if [f(1)]+sorted([f(2),f(3),f(4)]) not in Relations:   # C1       
            Relations.append([f(1)]+sorted([f(2),f(3),f(4)]))  # C1   
    if len(Relations) == 8:      
        R = matrix(QQ,0,8)                 # C2
        for i in range(8):                 # C2
            row = 8*[0]                    # C2
            row[Relations[i][0]-1]=1       # C2
            for j in [1,2,3]:              # C2
                row[Relations[i][j]-1]=-1  # C2
            R = R.stack(matrix(QQ,[row]))  # C2
        R_i=[R.stack(matrix(QQ,[[1]+(j-2)*[0]+[1]+(8-j)*[0]])) for j in range(2,9)]  #C3
        if math.prod([R_i[j].rank()-R.rank() for j in range(7)])!=0:#C4   
            Suitable.append(H)  # C4
        if [R_i[j].rank()-R.rank() for j in range(7)].count(0) == 1: #C5
            Galois_group_candidates.append(H)                        #C5
\end{lstlisting}

\begin{lstlisting}[language=Python,basicstyle=\ttfamily\scriptsize,caption=SageMath code used in the proof of Theorem~\ref{t2}.,label={lst:2}]
# transitive_subgroups - same as in Listing (*@\ref{lst:1}@*)
# For every H in transitive_subgroups, code segment C1 checks if H preserves all the relations in Rel which correspond to relations (*@\eqref{eq1053}@*) in the proof of Theorem (*@\ref{t2}@*). If H preserves all the relations in Rel, then the list Suitable is appended with the group H.
#Each relation $\alpha_i+\alpha_j+\alpha_k+\alpha_l=0$ is encoded as the set {i,j,k,l} in Rel.
Rel = Set([{1,2,3,4},{5,6,7,8},{1,2,5,6},{3,4,7,8},{1,3,5,7},{2,4,6,8}])
Suitable = []
for H in transitive_subgroups:
    FixesAllRelations = True
    for f in H.list():                                              # C1
        if not all([({f(i) for i in rel} in Rel) for rel in Rel]):  # C1
            FixesAllRelations = False
            break
    if FixesAllRelations:
        Suitable.append(H)
        
# For every H in Suitable, in view of the relations (*@\eqref{eq1053}@*), the orbit of $\alpha_1-\alpha_8$ is $\{\alpha_1-\alpha_8, \alpha_8-\alpha_1\}$ and the orbit of $\alpha_1+\alpha_8$ is 
   $\{\alpha_1+\alpha_8,\alpha_2+\alpha_7,\alpha_3+\alpha_6,\alpha_4+\alpha_5\}$.
Correct = True
for H in Suitable:
    if {(f(1),f(8)) for f in H.list()} != {(1,8),(8,1),(2,7),(7,2),(3,6),(6,3),(4,5),(5,4)}:
        Correct = False

# Computes the data presented in Table (*@\ref{table:tsS8pN6}@*)
for i in range(16):
    print(Suitable[i].structure_description(),Suitable[i].order(),Suitable[i].gens())
\end{lstlisting}

\end{document}